\documentclass[11pt]{article}
\usepackage{anysize}\marginsize{3.5cm}{3.5cm}{1.3cm}{2cm}
\makeatletter\@ifundefined{pdfpagewidth}{}{\pdfpagewidth=21.0cm\pdfpageheight=29.7cm}\makeatother % als Information für pdfTeX
\usepackage{amssymb,amsmath,graphicx}

\usepackage[latin1]{inputenc}

\makeatletter
\let\orig@item=\@item \def\@item[#1]{\orig@item[\rm #1]}
\renewenvironment{abstract}{\begin{quote}\footnotesize\textbf{\abstractname.}}
{\end{quote}\bigskip}
\newcommand\grant[1]{{\renewcommand\thefootnote{}\footnotetext{#1.}}}
\renewcommand\@seccntformat[1]{\csname the#1\endcsname.\enspace}
\renewcommand\section{\@startsection{section}{1}{\z@}
{-2\baselineskip plus 0.5\baselineskip minus 0.5\baselineskip}
{0.75\baselineskip plus 0.5\baselineskip}{\normalsize\bfseries\centering}}
\renewcommand\subsection{\@startsection{subsection}{2}{\z@}
{-2\baselineskip plus 0.5\baselineskip minus 0.5\baselineskip}
{0.75\baselineskip plus 0.5\baselineskip}{\normalsize\bfseries\centering}}
\renewcommand\paragraph{\@startsection{paragraph}{4}{\z@}{1\baselineskip}
{-0.5em}{\normalsize\bfseries}}
\let\origcaption=\caption \renewcommand\caption[1]{\parbox{0.66\textwidth}
{\origcaption{#1}}}

\renewcommand\@begintheorem[2]{\trivlist\item[\hskip\labelsep
{\bfseries#1 #2.}]\it}
\renewcommand\@opargbegintheorem[3]{\trivlist\item[\hskip\labelsep
{\bfseries#1 #2}] {\bfseries(#3).}\enspace\it\ignorespaces}
\makeatother

\newtheorem{satz}{Satz}[section]
\makeatletter\@addtoreset{equation}{satz}\makeatother

\newtheorem{theorem}[satz]{Theorem}

\newtheorem{cor}[satz]{Corollary}

\newtheorem{lemma}[satz]{Lemma}

\newtheorem{introtheorem}{Theorem}

\newtheorem{introcor}[introtheorem]{Corollary}

\newenvironment{definition}[1][\addtocounter{satz}{1}\bf Definition \thesatz]{\trivlist\item[\hskip\labelsep{\it #1.}]}{\endtrivlist}
\newenvironment{construction}[1][\addtocounter{satz}{1}\bf Construction \thesatz]{\trivlist\item[\hskip\labelsep{\it #1.}]}{\endtrivlist}
\newenvironment{remark}[1][\addtocounter{satz}{1}\bf Remark \thesatz]{\trivlist\item[\hskip\labelsep{\it #1.}]}{\endtrivlist}
\newenvironment{example}[1][\addtocounter{satz}{1}\bf Example \thesatz]{\trivlist\item[\hskip\labelsep{\it #1.}]}{\endtrivlist}

\newenvironment{proof}[1][Proof]{\trivlist\item[\hskip\labelsep{\it #1.}]}
{\hspace*{\fill}$\Box$\endtrivlist}

\newcommand\subjclass[1]{{\renewcommand\thefootnote{}\footnotetext{2000 
\textit{Mathematics Subject Classification:} #1.}}}
\newcommand\keywords[1]{{\renewcommand\thefootnote{}\footnotetext
{\textit{Keywords.} #1.}}}

\newcommand\engqq[1]{``#1''}

\renewcommand\geq{\geqslant}  % from amssymb
\renewcommand\epsilon{\varepsilon}
\renewcommand\phi{\varphi}

\renewcommand\O{{\cal O}}
\renewcommand\P{\mathbb P}

\newcommand\be{\begingroup\arraycolsep=0.13888em\begin{eqnarray*}}
\newcommand\ee{\end{eqnarray*}\endgroup}
\renewcommand\to{\longrightarrow}

\renewcommand\mapsto{\mapstochar\longrightarrow}

\newcommand\set[1]{\left\{#1\right\}}

\newcommand\with{\ \vrule\ }

\newlength\matrcolsep \matrcolsep=\arraycolsep

\newcommand\tline{\noalign{\vskip0.4ex}\hline\noalign{\vskip0.65ex}}

\newcommand\converges[2]{\mathop{\longrightarrow}\limits_{n\to\infty}}

\newcommand\N{\mathbb N}

\newcommand\R{\mathbb R}

\newcommand\Z{\mathbb Z}
\newcommand\C{\mathbb C}
\newcommand\newop[2]{\newcommand#1{\mathop{\rm #2}\nolimits}}
\newop\mult{mult}
\newop\NS{NS}
\newop\Amp{Amp}
\newop\Pic{Pic}
\newop\Bl{Bl} 
\newop\End{End} 
\newop\Nef{Nef}
\newop\Mov{Mov}
\newop\vol{vol}
\newop\codim{codim}
\newop\BigCone{Big} 
\newop\interior{int}

\newop{\rang}{rang}
\newop{\dv}{div}
\newop{\id}{id}
\newop{\ord}{ord}
\newop{\Div}{Div}
\newop{\Proj}{Proj}
\newop{\im}{im}
\newop{\Face}{Face}
\newop{\conv}{conv}
\newop{\exc}{exc}
\newop{\dom}{dom}
\newop{\Supp}{Supp}
\newop{\Vol}{Vol}
\newop{\Neg}{Neg}
\newop{\Cox}{Cox}

\newop{\Null}{Null}
\newop{\NE}{\overline{{NE}}}
\newop{\Eff}{\overline{{Eff}}}

\begin{document}

   \title{On the polyhedrality of global Okounkov bodies}
   \author{David Schmitz and Henrik Sepp\"anen }
   \date{
   }
   \maketitle
   \thispagestyle{empty}

\grant{ The first author was supported by DFG grant BA 1559/6-1. The second author was supported by the DFG Priority Programme 1388 ``Representation 
Theory"}

%\toappearin[5.5cm]{ <Zeitschrift> }

%*****************************************************************************

\begin{abstract}
   We prove that the existence of a finite Minkowski base for
   Okounkov bodies on a smooth projective variety with respect to an admissible flag
   implies rational polyhedrality of the global Okounkov body. As an application of this general result,
   we deduce that the global Okounkov body of a surface with finitely generated pseudo-effective cone with 
   respect to a general flag is rational polyhedral. We give an alternative proof for this fact 
   which  recovers the generators more explicitly. We also prove the rational
   polyhedrality of global Okounkov bodies in the case of certain homogeneous 3-folds 
   using inductive methods.
\end{abstract} 

%*****************************************************************************

\section*{Introduction}
During the last couple of years, the construction of Okounkov bodies of pseudo-effective divisors on a variety $X$ has gained quite a lot of attention. Following an idea of Okounkov (\cite{ok}), it has been formally introduced independently by Kaveh and Khovanskii (\cite{kk}), and Lazarsfeld and Musta\c t\u a (\cite{lm}). For details on the construction we refer to these two seminal papers. Noticed from the beginning, the most prominent feature of the Okounkov body $\Delta_{Y_\bullet}(D)$ of a given pseudo-effective divisor $D$ with respect to some admissible flag $Y_\bullet$ is the fact that its volume is independent of the chosen flag and recovers the volume of the divisor $D$. 

The construction for one divisor $D$ can be extended to the global situation of all divisors on a given variety as was shown in \cite{lm}. More concretely, there exists a closed convex cone $\Delta_{Y_\bullet}(X)$ in the direct product $\R^n\times N^1(X)_\R$ such the fiber with respect to the second projection over each big divisor class $[D]$ is the body $\Delta_{Y_\bullet}(D)$, i.e.,
$$
	pr_2^{-1}(D)\cap\Delta_{Y_\bullet}(X)=\Delta_{Y_\bullet}(D) \times \{[D]\}.
$$

Determining Okounkov bodies in general as well as describing their geometric properties is notoriously hard, and the situation for global Okounkov bodies is even worse. It was proven by Anderson-K\"uronya-Lozovanu (\cite{akl}) on one hand, and by the second author on the other hand (cf. \cite{sep1}) in 2012 that ample line bundles always admit an admissible flag for which the Okounkov body is polyhedral. In fact, the first mentioned authors also prove this for 
semi-ample line bundles. For the global body, however, less in this regard is known.
 For Mori dream spaces, Okawa in \cite{oka} gives conditions on a flag which would imply 
that the corresponding global Okounkov body is rational polyhedral. Toric varieties have
polyhedral global Okounkov bodies with respect to a torus invariant flag \cite[Proposition 6.1]{lm}. This also holds for projectivizations of rank two toric vector bundles (\cite{gon}), and as was shown in \cite{pet}, for rational complexity-one $T$-varieties. In \cite{sep}, the 
second author of this paper shows that a homogeneous surface with a rational polyhedral pseudo-effective cone admits a rational polyhedral global Okounkov body. To our knowledge no other cases are known so far.

One approach to the determination of Okounkov body of big divisors on a given variety $X$  has been presented in \cite{dp} and \cite{psu}. The basic idea is to find elementary \engqq{building-blocks} from which all Okounkov bodies of pseudo-effective divisors on $X$ can be constructed as Minkowski sums. For the cases of general flags on smooth surfaces and torus-invariant flags on toric varieties it was shown respectively in the above papers that there exists such a \engqq{Minkowski base}. We recall the precise definition in section \ref{S: Mink}. 

In the present note we investigate the consequences of the existence of a Minkowski base for the shape of the global Okounkov body. Concretely, we prove the following.
\begin{introtheorem}\label{introth}
	Let $X$ be a smooth projective variety and let $Y_{\bullet}$ 
	be an admissible flag such that $X$ admits a Minkowski base $D_1,\ldots,D_r$ with
	respect to $Y_\bullet$ whose corresponding Okounkov bodies $\Delta_{Y_\bullet}(D_i)$ 
	are rational polyhedral. Then the global Okounkov body $\Delta_{Y_\bullet}(X)$ 
	is rational polyhedral.
	
	More concretely, it is spanned by the set of vectors
	$$
	\bigcup_i\set{(x,[D_i])\with x \mbox{ vertex of } \Delta_{Y_\bullet}(D_i)}.
	$$
\end{introtheorem}
We thus obtain concrete vectors generating the global Okounkov body in terms of the Minkowski base and the vertices of the corresponding indecomposable bodies. It turns out that in general the set of generators need not be minimal. By the above mentioned result from \cite{dp}, the theorem yields the
following.
\begin{introcor}
	Let $X$ be a smooth projective surface with rational polyhedral effective cone. 
	Then the global Okounkov body $\Delta_{Y_\bullet}(X)$ with respect to a general flag is rational
	polyhedral. 
\end{introcor}
As in the theorem, the generating set consists of vectors given by corners of Okounkov bodies of all Minkowski base elements. This result can be improved by methods not depending on the Minkowski base construction in the sense that we  obtain explicitly a set of vectors generating the global Okounkov body. More concretely, in section \ref{s:surface} we prove the following.
\begin{introtheorem}
	The cone $\Delta_{Y_\bullet}(X)$ is the closed convex cone generated by the vectors 
	$$
	((0,0), [D_i]), ((0,P_i \cdot A),[ D_i]), ((1,0), [A]), \quad i=1,\ldots, r,
	$$
	where the $D_i$ are the generators of extremal rays of BKS-chambers, $P_i$ are their positive part, 
	and $A$ is the numerical class of the curve $Y_1$.
\end{introtheorem}
The rational polyhedrality of the global Okounkov body in this setting does not come as a surprise. In fact, it appears to be considered folklore knowledge since the nice description of Okounkov bodies of divisors on surfaces in \cite{KLM}. However, we are not aware of a proof, let alone an explicit description of the generators, having been given so far.

In the final section of this note, we use the above result on surfaces to obtain a similar one for  homogeneous 3-folds.

Throughout this paper we work over the complex numbers.
 
\bigskip
{\small\noindent {\bf Acknowledgements.} } The authors would like to thank Thomas Bauer and Alex K\"uronya  for helpful discussions and suggestions. Additionally, we express our gratitude towards the organizers of the workshops \engqq{Convex bodies and Representation Theory}, Megumi Harada, Kiumars Kaveh, and Askold Khovanskii; and \engqq{Positivity of linear series and vector bundles}, S\'andor Kov\'acs, Alex Küronya, and Tomasz Szemberg, at BIRS, during which this collaboration was initiated.

%
%%
%******************************************************************************
\section{Minkowski chambers}\label{S: Mink}
In this section we recall the definition of a Minkowski base and introduce a 
crucial construction for the proof of Theorem \ref{introth}, the Minkowski chamber 
decomposition. 

Let $X$ be a smooth projective variety over the complex numbers and let 
$Y_{\bullet} : X = Y_{0} \supseteq Y_{1} \supseteq \dots \supseteq Y_{n-1} \supseteq Y_{n} = \{pt\}$
be an admissible flag such that $X$ admits a Minkowski base $D_1,\ldots,D_r$ consisting of pseudo-effective divisors, with respect to $Y_\bullet$
in the following sense.
\begin{definition}
	A finite collection $\set{D_1,\ldots,D_R}$ of pseudo-effective divisors on a smooth projective variety $X$ is a \emph{Minkowski base} if 
	\begin{itemize}
		\item
			For any pseudo-effective divisor $D$ on $X$ there exist non-negative 
			numbers $a_1\ldots,a_r$ such that
			$$
				D=\sum a_iD_i,\qquad\mbox{and }\qquad \Delta_{Y_\bullet}(D)=\sum
				 a_i\Delta_{Y_\bullet}(D_i), and
			$$
		\item	
			the $\Delta_{Y_\bullet}(D_i)$ are indecomposable in the sense 
			of Minkowski sums.
	\end{itemize}
\end{definition}
 
Note that we allow here as Minkowski summands $\Delta_{Y_\bullet}(D_i)$ which are just affine points, i.e., which correspond to fixed divisors $D_i$, as opposed to the approach in \cite{psu} where Okounkov bodies of non-movable divisors where represented as translates of movable ones. In fact, in section \ref{s:globalOKB} we interpret the translations there as a Minkowski sum with a linear combination of the affine points $\Delta_{Y_\bullet}(D_i)$ representing the negative part of a divisor. 

We will use the following.
\begin{construction}\emph{(Minkowski chambers).}
The Minkowski base induces a chamber decomposition of the pseudo-effective cone into simplicial cones, each spanned by exactly $\rho$ of the Minkowski base elements and such that no Minkowski base element apart from the spanning ones is contained in any chamber. This decomposition is obtained as follows. If there is a Minkowski base element $\gamma$ not contained in one of the rays spanning $\Eff(X)$ then decompose $\Eff(X)$ into subcones spanned by the sides of $\Eff(X)$ and the ray spanned by $\rho$. Repeat the process for each subcone until no Minkowski base elements apart from spanning ones lie in each cone. Now, we can pass to a triangulation of each subcone into simplicial cones without having to add any new rays. We call the resulting subcones \emph{Minkowski chambers} of $\Eff(X)$.
\end{construction}
Note that the chambers are rational cones (with generators the corresponding Minkowski base elements) and that the coefficients of the Minkowski decomposition are unique and vary linearly on the closure of each chamber, provided we allow only decompositions with respect to the base elements spanning the chamber. The chamber decomposition itself need not be unique but any triangular decomposition in the above sense will do for our purpose.

\begin{example}\label{ex: Bl2}

Let $\pi:X\to\P^2$ be the blowup in two points $p_1,p_2$ with exceptional divisors $E_1,E_2$. Denote by $H$ the pullback $\pi^\ast(\mathcal O_{\P^2}(1))$. The pseudo-effective cone $\Eff(X)$ is spanned by the classes $E_1,E_2$ and $H-E_1-E_2$. Taking as a flag a general member $C$ of the ample class $3H-E_1-E_2$ and a general point on it, following the algorithm in \cite{dp}  we obtain the Minkowski base consisting of the elements $3H-E_1-E_2, H, 3H-E_1, 3H-E_2, 2H-E_1-E_2,  H-E_1, H-E_2$. The following figure illustrates the corresponding chamber decomposition.

	\begin{figure}[ht]
\centering
\unitlength 1mm 
\linethickness{0.4pt}
\ifx\plotpoint\undefined\newsavebox{\plotpoint}\fi 
\begin{picture}(75,50)(0,10)
\includegraphics{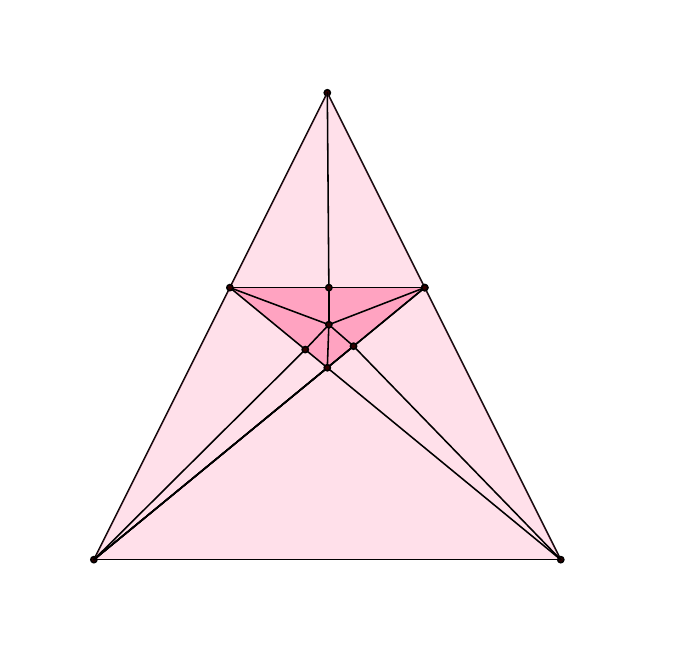}
\end{picture}
\caption{Minkowski chamber decomposition of $\Eff(X)$\label{fig}}
\end{figure}

\end{example}

%***********************************************************************************
%
\section{The global Okounkov body}\label{s:globalOKB}

Now, consider the global Okounkov body $\Delta_{Y_\bullet}(X)\subseteq \R^n\times N^1(X)_\R$  of $X$ with respect to the flag $Y_\bullet$, which for each big divisor $D$ satisfies the condition
$$
	pr_2^{-1}(D)\cap\Delta_{Y_\bullet}(X)=\Delta_{Y_\bullet}(D)\times \set{[D]}.
$$
We can now prove the following. 

\begin{theorem}\label{th}
	Let $X$ be a smooth projective variety and let 
	$Y_{\bullet} : X = Y_{0} \supseteq Y_{1} \supseteq \ldots
	\supseteq Y_{n-1} \supseteq Y_{n} = \{pt\}$ 
	be an admissible flag such that $X$ admits a Minkowski base $D_1,\ldots,D_r$ with
	respect to $Y_\bullet$ whose
	corresponding Okounkov bodies $\Delta_{Y_\bullet}(D_i)$ are rational 
	polyhedral. Then the global Okounkov body $\Delta_{Y_\bullet}(X)$ is rational polyhedral.
	
	More concretely, it is spanned by the set of vectors
	$$
	\bigcup_i\set{(x,[D_i])\with x \mbox{ vertex of } \Delta_{Y_\bullet}(D_i)}.
	$$
\end{theorem}

\begin{proof}
Pick a Minkowski chamber decomposition as described in Construction 1.2. 
We prove that
$$
		pr_2^{-1}(\mathcal C)\cap\Delta_{Y_\bullet}(X)
$$
is rational polyhedral for each Minkowski chamber $\mathcal C$ in the decomposition, from which
 it follows by a general combinatorial result (\cite[Lemma A.1]{oka})
that the cone $pr_2^{-1}(\Eff(X))\cap\Delta_{Y_\bullet}(X)$ is already rational polyhedral. 
This implies that in fact the whole global body $\Delta_{Y_\bullet}(X)$ is rational polyhedral.

Let $\mathcal C$ be a chamber in the above decomposition spanned by Minkowski base elements $D_1,\ldots,D_\rho$. Then for every $D=\sum_{i=1}^\rho a_iD_i$ in $\mathcal C$ we have
$$
	pr_2^{-1}(D)\cap\Delta_{Y_\bullet}(X)=\Delta_{Y_\bullet}(D)=\sum a_i\Delta_{Y_\bullet}(D_i),
$$
in other words, $pr_2^{-1}(\mathcal C)\cap\Delta_{Y_\bullet}(X)$ is spanned by 
the set of vectors
$$
	\bigcup_i\set{(x,[D_i])\with x \mbox{ vertex of } \Delta_{Y_\bullet}(D_i)},
$$
hence it is in fact rational polyhedral as soon as the bodies $\Delta_{Y_\bullet}(D_i)$ are.
Since this is the case for all Minkowski base elements $D_i$ by assumption, the cone $pr_2^{-1}(\Eff(X))\cap\Delta_{Y_\bullet}(X)$ is rational polyhedral as well. 
\end{proof}

%********************************************************************

In order to apply the above theorem to the cases studied in \cite{psu} and \cite{dp} first we make sure that the results obtained there for toric varieties and surfaces, respectively, yield Minkowski bases also in our sense. 

All we need to do is to augment the Minkowski base consisting of movable divisors constructed in the papers by the supports of negative parts in the Zariski decomposition of big divisors. These correspond in the considered cases to finitely many effective divisors. Decomposing the Zariski decomposition $D= P+N$ of a pseudo-effective divisor $D$ as  
$$
	D= P+N =\sum a_i P_i + \sum b_j N_j,
$$
with respect to Minkowski base elements, yields the decomposition of the Okounkov body
\be
	\Delta_{Y_\bullet}(D)&=& \sum a_i\Delta_{Y_\bullet}(P_i) + \sum b_j \Delta_{Y_\bullet}(N_j)\\
		       &=& \sum a_i\Delta_{Y_\bullet}(P_i) + \sum b_j \nu_\bullet(s_j)\\
		       &=& \sum a_i\Delta_{Y_\bullet}(P_i) + \nu_\bullet(s_1^{b_1}\cdot\ldots\cdot s_m^{b_m}),\\
\ee
where by abuse of notation $\nu_\bullet(s_j)$ stands for the normalized valuation vector of a section $s_j\in H^0(X,\O_X(mN_j))$ for large enough $m$. Now, the last summand exactly gives the translation $\phi$ from \cite[Definition 1.3]{psu}.

We thus get the following.

\begin{cor}\label{C: surface}
	Let $X$ be a smooth projective surface with rational polyhedral effective cone. 
	Then the global Okounkov body $\Delta_{Y_\bullet}(X)$ with respect to a general flag is rational
	polyhedral.
\end{cor}

\begin{proof}
	This follows from Theorem \ref{th} together with \cite[Theorem]{dp} and the 
	above argumentation.
\end{proof}
 By the same argument and citing \cite[Theorem 3.1]{psu}, we also obtain the following, which was already noted in \cite[Proposition 6.1]{lm}. 
\begin{cor}
	Let $X$ be a toric variety. 
	Then the global Okounkov body $\Delta_{Y_\bullet}(X)$ with respect to a torus-invariant flag is
	rational polyhedral.
\end{cor}
Note that Theorem \ref{th} in addition enables us in both cases to recover generators of the global Okounkov body. However, in general the generating set is not minimal. Some of the rays spanned need not be extremal in $\Delta_{Y_\bullet}(X)$. In the following section,  we present a direct proof for the fact that the global Okounkov body is rational polyhedral in the setting of Corollary \ref{C: surface} which explicitly gives a set of generating vectors, which in general is smaller than the set given by Theorem \ref{th}.

%**************************************************************************************
%

\section{Generators of the global body of a surface}\label{s:surface}

Let $X$ be a smooth projective surface admitting a rational polyhedral 
pseudo-effective cone $\overline{\mbox{Eff}}(X)$. We consider the chamber decomposition introduced in \cite{bks}. It follows from the assumptions that there are finitely many BKS-chambers, whose closures are all rational polyhedral. Let $\{D_1,\ldots, D_r\}$ the union of the generators of all these closures of BKS chambers, and let $D_i=P_i+N_i$ be 
the Zariski decomposition of $D_i$.

Let us make explicit what we mean by a general flag on $X$. Let 
$\{E_j\}_{j=1}^\infty$ be an enumeration of the 
integral divisors in $\overline{\mbox{Eff}}(X)$ admitting a 
Zariski decomposition $E_j=P_j+N_j$ where both $P_j$ and $N_j$ 
are integral. Let $s_j \in H^0(X, \mathcal{O}_X(N_j))$ be the 
defining section for $N_j$. 

We now define an admissible flag on $X$ as follows. Let $A$ be a big and semi-ample 
divisor, and let $s_A \in H^0(X, \mathcal{O}_X(A))$ 
be a section such that the zero set $Z(s_A)$ is an irreducible curve which 
neither lies in the union of 
the base loci $B(D_i), \,  i=1,\ldots, r$, nor  in the union of 
the subvarieties $Z(s_j), \, j \in \N$.  Define $Y:=Z(s_A)$, 
and let $p \in Y$ be a regular point which does lie in the union of the $Z(s_j)$ 
and the $B(D_i)$.

Let $v$ be the valuation on $\C(X)^*$ defined by the admissible 
flag $$X \supset Y\supset\{p\} $$ and let $\Delta_{Y_\bullet}(X)$ be 
the associated global Okounkov body.
%\begin{align*}
%S_{Y_\bullet}(\Sigma_P)&:=\{(v(s), [D]) \mid s \in H^0(X, \mathcal{O}_X(D)) \setminus \{0\}, [D] \in 
%\mbox{int}(\Sigma_P)\}\\
%&\subseteq \N_0^2 \times N^1(X)_\Z.
%\end{align*}
%Let $C_{Y_\bullet}(\Sigma_P) \subseteq \R^2 \oplus N^1(X)_\R$ be the closed convex cone 
%generated by the semigroup $S_{Y_\bullet}(\Sigma_P)$.  

\begin{lemma} \label{L: genincone}
For each generator $D_i$, with Zariski decomposition $D_i=P_i+N_i$, 
the inclusion 
\begin{align*}
\{0\}\times [0, D_i \cdot A] \times  \{[D_i]\} \subseteq \Delta_{Y_\bullet}(X)
\end{align*}
holds.
\end{lemma}

\begin{proof}
By the choice of the point $p$, the defining section of $mN_i$, for $m$ big enough, 
has value zero. Hence, $\Delta_{Y_\bullet}(D_i)=\Delta_{Y_\bullet}(P_i)$.
Moreover, since $p$ does not lie in the base locus of $P_i$, we have 
$(0,0) \in \Delta_{Y_\bullet}(P_i)$.
Finally, since $P_i$ is nef, we also have $(0,P_i \cdot A) \in \Delta_{Y_\bullet}(P_i)$ 
(cf. \cite[Lemma 4.1]{sep}). 
Thus, $((0,0), [D_i]) \in \Delta_{Y_\bullet}(X)$ and $(0,(P_i \cdot A),[D_i]) \in \Delta_{Y_\bullet}(X)$.
The claim now follows from the convexity of  $\Delta_{Y_\bullet}(X)$.
\end{proof}

\begin{theorem}
The cone $\Delta_{Y_\bullet}(X)$ is the closed convex cone 
generated by the vectors 
$((0,0), [D_i]), ((0,P_i \cdot A),[ D_i]), ((1,0), [A]), \quad i=1,\ldots, r$. 
\end{theorem}

\begin{proof}
Let $D$ be a big divisor, and let $s \in H^0(X, \mathcal{O}_X(D))$ be 
a nonzero section with $v(s)=(a, b)$. Then $\zeta:=s/s_A^a \in H^0(X, \mathcal{O}_X(E))$, 
where  $E:=D-aA$. The numerical equivalence class of the divisor $E$ lies in the closure of a BKS-chamber $\Sigma_P$, for some big and nef divisor $P$.  Let 
$E=P_E+N_E$ be the Zariski decomposition of $E$. 
The (numerical equivalence class of the) divisor $E$ can then be written as a linear combination 
\begin{align*}
[E]=t_1[D_1]+\cdots +t_r[D_r], \quad t_1,\ldots, t_r \geq 0, 
\end{align*}
and where $t_i \neq 0$ only for those $[D_i]$ lying on the boundary of $\Sigma_P$. 
By the linearity of Zariski decompositions on closures of BKS-chambers 
(\cite[Proposition 2.3]{dp}), 
the positive part $P_E$ then decomposes as 
\begin{align*}
[P_E]=t_1[P_1]+\cdots +t_r[P_r]. 
\end{align*}

Now, let $m \in \N$ be so big that $mP_E$ and $mN_E$ are both integral 
divisors. The section $\zeta^m \in H^0(X, \mathcal{O}_X(mE))$ then 
factorizes uniquely as $$\zeta^m=\eta \sigma,$$
with $\eta \in H^0(X, \mathcal{O}_X(mP_E))$, and $\sigma \in H^0(X, \mathcal{O}_X(mN_E))$.
Moreover, $(0,mb)=v(\zeta^m)=v(\eta)$, 
by the choice of the point $p \in Y$,
since $\sigma$ is one of the $s_j$.
Hence $mb \in [0,mP_E \cdot A]$, so that 
\begin{align*}
mb=cm \sum_{i=1}^r t_i P_i \cdot A,
\end{align*}
for some $c \in [0,1]$. Thus,
\begin{align*}
((0,mb), m[E])=cm\sum_{i=1}^rt_i((0,P_i \cdot A), [D_i])+m(1-c)\sum_{i=1}^rt_i((0,0), [D_i]),
\end{align*}
from which it follows that $((a, b), [D])$ lies in the closed convex cone 
generated by the vectors 
$((0,0), [D_i]), ((0,P_i \cdot A), [D_i]), ((1,0), [A]), \quad i=1,\ldots, r$.

Hence, the global Okounkov body $\Delta_{Y_\bullet}(X)$ is 
contained in the closed convex cone generated by these vectors. 
In view of  Lemma \ref{L: genincone}, equality thus holds.
\end{proof}

\begin{remark}
If $X$ is a homogeneous surface, i.e., carrying a transitive action of a connected 
algebraic group, every effective divisor is nef. In the case when the pseudo-effective 
cone is rational polyhedral, there is then only one BKS-chamber. 
The proof of the above theorem in this case then yields the proof of 
\cite[Theorem 4.3]{sep}. 
\end{remark}

%*****************************************************************************************

\begin{example}
Consider the situation of Example \ref{ex: Bl2}. Note that $\Eff(X)$ is finitely generated and we can apply the theorem once we fix a permitted flag. The divisor $H$ is big and semi-ample, in particular, its general member $C$ is an irreducible curve not contained in any negative part of big divisors on $X$. Fix a general point $x$ on $C$ to obtain a flag $(C,x)$.

 Note that $\Eff(X)$ decomposes into five BKS-chambers: the nef chamber and the four chambers corresponding to the big and nef divisors $H, 2H-E_1, 2H-E_2, 2H-E_1-E_2$. The classes $\xi_i$ in the theorem are given in the following table.

\begin{table}[ht]
      \centering
      \newlength\extraarrayskip \extraarrayskip=2ex
      \begin{tabular}{cc}\tline
      Chamber & $\qquad$Generators$\qquad$    \\\tline
            $\Sigma_A$ & $H,\ H-E_1,\ H-E_2$ \\[\extraarrayskip]
            $\Sigma_{H}$ & $H,\ E_1,\ E_2$ \\[\extraarrayskip]
            $\Sigma_{2H-E_1}$ &$H,\ H-E_1,\ E_2$ \\[\extraarrayskip]
            $\Sigma_{2H-E_2}$ & $H,\ H-E_2,\ E_1$ \\[\extraarrayskip]
            $\Sigma_{2H-E_1-E_2}$ & $ H-E_1,\ H-E_2,\ H-E_1-E_2$ \\[\extraarrayskip]
            
            \tline
      \end{tabular}
      \caption{\label{fig:generators}%
         Generators of the BKS-chambers}
   \end{table}

The global Okounkov body of $X$ with respect to the flag $(C,x)$ consequently is generated by the following vectors:
\be
	 ((0,0),[H]),&&\quad ((0,1),[H]),\\
	((0,0),[H-E_1]),&&\quad ((0,1),[H-E_1]),\\
	((0,0),[H-E_2]),&&\quad ((0,1),[H-E_2]),\\
	((0,0),[E_1]),&&\quad ((0,0),[E_2]),\\
	((0,0),[H-E_1-E_2]),&&\quad((1,0),[H]).\\
\ee
To illustrate this result, let us check how the fiber $pr_2^{-1}(D)\cap\Delta_{Y_\bullet}(X)$ over a given big divisor $D$, which should be just the Okounkov body $\Delta_{Y_\bullet}(D)\times\set{D}$, arises as positive linear combination of the generators. Pick $D$ to be the ample divisor $3H-E_1-E_2$. The Okounkov body is the convex hull of the points $(0,0), (0,3), (1,2), (2,0)$. Now the corresponding vectors   in $\Delta_{Y_\bullet}(X)$ arise as follows:
\be
	((0,0),[D]) &=&  ((0,0), [H]) + ((0,0), [H-E_1]) +  ((0,0), [H-E_2])\\
	((0,3),[D]) &=&  ((0,1), [H]) +  ((0,1), [H-E_1]) +  ((0,1), [H-E_2])\\
	((1,2),[D]) &=&  ((1,0), [H]) + ((0,1), [H-E_1]) + ((0,1), [H-E_2])\\
	((2,0),[D]) &=&  2\cdot ((1,0),[H]) + ((0,0), [H-E_1-E_2]).
\ee

\end{example}

\section{Homogeneous 3-folds}

Let $X$ be a 3-dimensional homogeneous projective variety, i.e., $X$ carries a 
transitive action of a complex connected algebraic group $G$,  and let $Y_1 \subseteq X$ be 
an smooth irreducible very ample divisor on $X$ with defining section 
$s_{Y_1} \in H^0(X, \mathcal{O}(Y_1))$. Assume further that both $X$ and the 
divisor $Y_1$ have rational polyhedral pseudo-effective cones.

Since $X$ is homogeneous, every effective divisor on $X$ is 
nef (cf. \cite[Example 1.4.7.]{laz}). Let therefore $D_1, \ldots, D_r$ be integral 
nef divisors generating the pseudo-effective cone $\overline{\mbox{Eff}}(X)$.
Let $Q \subseteq \overline{\mbox{Eff}}(Y_1)$ be the closed convex subcone 
generated by the restrictions to $Y_1$ of the divisors $D_1,\ldots, D_r$.
Let 
\begin{align}
Y_3 \subseteq Y_2 \subseteq Y_1\label{E: surfaceflag}
\end{align}
 be an admissible flag on $Y_1$, as in section \ref{s:surface}, and let $Y^1_\bullet$ 
denote this flag. Since the global Okounkov body $\Delta_{Y^1_\bullet}(Y_1)$ of $Y_1$ 
with respect to $Y^1_\bullet$ is rational polyhedral, by assumption, the closed convex 
subcone $p_2^{-1}(Q) \cap \Delta_{Y^1_\bullet}(Y_1)$, where $p_2$ denotes the projection 
onto the second factor, is also rational polyhedral.

Let $v: \C(X)^* \rightarrow \Z^3$ be the valuation defined by the flag $Y_\bullet$,
$Y_3 \subseteq Y_2 \subseteq Y_1 \subseteq Y_0:=X$, and let $v^1: \C(Y)^* \rightarrow \Z^2$ be 
the valuation defined by the truncated flag \eqref{E: surfaceflag}.  We now define the 
semigroups
\begin{align*}
S&:=\{(v(s), [D]) \in \N_0^3 \times N^1(X)_\Z \mid s \in H^0(X, \mathcal{O}(D)), v_1(s)=0\},\\
S_1&:=\{(v^1(s), [D \cdot Y_1]) \in \N_0^2 \times N^1(Y_1)_\Z \mid 
s \in H^0(Y_1, \mathcal{O}(D \cdot Y_1))\},
\end{align*}
as well as the morphism
\begin{align*}
q: S \rightarrow S_1, \quad q(v(s), [D]):=(v_2(s), v_3(s), [D  \cdot Y_1])
\end{align*}
of semigroups. Then $q$ extends uniquely to a linear map 
$\R^2 \oplus N^1(X)_\R \rightarrow \R^2 \oplus N^1(Y_1)_\R$, 
which we will also denote by $q$. Here we have embedded $\R^2$ into $\R^3$ by the linear 
map $(x, y) \mapsto (0,x,y)$. If $C(S) \subseteq \R^3 \oplus N^1(X)_\R$ and 
$C(S_1) \subseteq \R^3 \oplus N^1(Y_1)_\R$ are the closed convex cones generated by the 
semigroups $S$ and $S_1$, respectively, the inclusion 
\begin{align*}
\Delta_{Y^1_\bullet}(\Delta \cdot Y_1) \times \{0\} \subseteq \Delta_{Y_\bullet}(D), 
\quad [D] \in \overline{\mbox{Eff}}(X),  
\end{align*}
which follows by the same proof as \cite[Lemma 4.1]{sep},
implies the equality
\begin{align*}
C(S)&=q^{-1}(C(S_1)) \cap \left((\R_{\geq 0})^2 \times \overline{\mbox{Eff}}(X)\right)\\
&=q^{-1}(p_2^{-1}(Q) \cap \Delta_{Y^1_\bullet}(Y_1)) \cap 
\left((\R_{\geq 0})^2 \times \overline{\mbox{Eff}}(X) \right).
\end{align*}
In particular, $C(S)$ is a rational polyhedral cone. Let $w_1,\ldots, w_k \subseteq C(S)$ be 
integral generators of $C(S)$.

\begin{theorem} \label{T: 3homokb}
Let $X$ be a homogeneous 3-fold admitting an admissible flag $Y_\bullet$ 
as above. Then the global Okounkov body $\Delta_{Y_\bullet}(X)$ 
of $X$ with respect to the flag $Y_\bullet$
is generated by the vectors $(v(s_{Y_1}), [Y_1]), w_1,\ldots, w_k$.
\end{theorem}

\begin{proof}
Clearly, all these vectors belong to $\Delta_{Y_\bullet}(X)$. On the other hand, 
let $E$ be an effective integral divisor on $X$, and let $s \in H^0(X, \mathcal{O}(E))$ 
be a nonzero section. If $s$ vanishes to order $a$ along $Y_1$, the 
section $\zeta:=s/s_{Y_1}^a \in H^0(X, \mathcal{O}(E-aY_1))$ vanishes to 
order $0$ along $Y_1$, i.e., $v_1(\zeta)=0$. Hence, $(v(\zeta), [E-aY_1]) \in S$, 
so that there exist $t_1,\ldots, t_k \geq 0$, such that $(v(\zeta), [E-aY_1])$ can be 
written as the linear combination
$(v(\zeta), [E-aY_1])=t_1w_+\cdots+t_kv_k$. It follows that
\begin{align*}
(v(s), [E])=(v(s_{Y_1}), [Y_1])+t_1w_1+\cdots+t_kw_k.
\end{align*}
This shows that the vectors  $(v(s_{Y_1}), [Y_1]), w_1,\ldots, w_k$ generate the 
cone  $\Delta_{Y_\bullet}(X)$.
\end{proof}

\begin{cor}
Let $G$ be a complex reductive group, and let $P \subseteq G$ be a parabolic 
subgroup, such that the flag variety $X=G/P$ is 3-dimensional. 
Then $X$ admits a flag $Y_\bullet$ for which the global Okounkov 
body $\Delta_{Y_\bullet}(X)$ is rational polyhedral.
\end{cor}

\begin{proof}
Let $B \subseteq P$ be a Borel subgroup of $G$ with Lie algebra $\mathfrak{b}$, and 
let $\mathfrak{h} \subseteq \mathfrak{b}$ be a Cartan subalgebra together 
with a choice $\Phi^+$ of positive roots which exhibits $\mathfrak{b}$ as the 
direct sum of $\mathfrak{h}$ and the sum of the positive root spaces. 
Let $2\rho_P \in \mathfrak{h}^*$ be the sum of the fundamental weights 
which are not attached to simple roots of the Levi factor $L_P$ of P. 
It is well-known that the canonical line bundle $K_X$ is given by the line 
bundle $L_{-2\rho_P}$ induced from the character $\exp(-\rho_P)$ of $P$, 
where the line bundle $L_{\rho_P}$ is very ample. Hence, $X$ is 
a Fano variety. Since $L_{\rho_P}$ is very ample, we can now choose 
a smooth, and hence irreducible, divisor $Y_1 \subseteq X$ such that 
$\mathcal{O}(Y_1)=L_{\rho_P}$. By the adjunction formula we now have
\begin{align*}
K_{Y_1}=K_X+Y_1,
\end{align*}
so that $\mathcal{O}(K_{Y_1})=L_{-\rho_P}\mid_{Y_1}$, which shows that 
$Y_1$ is also a Fano variety. Hence, the pseudo-effective cone is rational polyhedral. 
Thus, $X$ and the divisor $Y_1$ satisfy the assumptions of Theorem \ref{T: 3homokb}
\end{proof}

\begin{remark}
The 3-dimensional flag varieties $G/P$, where $G$ is a reductive complex 
group, and $P \subseteq G$ a parabolic subgroup, are the following: 
the full flag variety $\C^3$, i.e., the variety of all flags of 
subspaces $V_1 \subseteq V_2 \subseteq \C^3$, where $\mbox{dim} V_1=1$ and 
$\mbox{dim} V_2=2$; the products $\mathbb{P}^1 \times \mathbb{P}^1 \times \mathbb{P}^1$ 
and $\mathbb{P}^1 \times \mathbb{P}^2$; and the Grassmannian of all Lagrangean subspaces of 
$\C^4$ with respect to the standard symplectic form. 
\end{remark}

Finally, as an example of homogeneous 3-folds satisfying 
the above conditions, we consider certain abelian varieties.
 By \cite{b:abelian} the condition that $\Eff(X)$ 
be rational polyhedral is equivalent to $X$ being isogenous 
to a product of non-isogenous abelian varieties of Picard 
number 1. This means $X$ is either simple of Picard number 
one, or $X$ is isogenous to $Y_1 \times E$ for an abelian 
surface $Y_1$ with $\rho(Y_1)=1$ and an elliptic curve $E$, 
or $X$ is isogenous to $E_1\times E_2\times E_3$ for 
non-isogenous elliptic curves $E_i$. 

We prove the following
\begin{cor}
	Let $X$ be a abelian 3-fold such that $\Eff(X)$ is rational 
	polyhedral. Then there exists an admissible flag $Y_\bullet$
	such that the global Okounkov body $\Delta_{Y_\bullet}(X)$
	is rational polyhedral. 
\end{cor}
\begin{proof}
	If $X$ is simple, then let $H$ be an ample generator of $\N^1(X)$, 
	and there is a flag such that $\Delta(H)$ is rational polyhedral. 
	Then the global Okounkov body $\Delta(X)$ is generated by the vectors
	$\set{(x,[D])\with x \mbox{ a vertex of } \Delta(H)}$.
	
	If $X$ is non-simple, then it is isogenous to a product 
	$E\times M$ of an elliptic curve $E$ and an abelian surface $M$ 
	of Picard number 1 or 2. In either case, $\Eff(M)$ is rational polyhedral.
	
	Note that as in Poincaré's complete reducibility theorem 
	(\cite[Theorem 5.3.7]{LB}) the isogeny 
	$$
		E\times M \to X
	$$
	is just given by addition. The image $Y_1$ of the abelian subvariety
	$\{0\}\times M$ is again an abelian subvariety, isogenous to $M$.
	By \cite[Lemma 3.1]{b:abelian}, $Y_1$ has rational polyhedral effective cone as well.
	Constructing a flag as in the theorem thus yields a rational polyhedral
	global Okounkov body.
	
\end{proof}

  \vskip .2cm
 David Schmitz,
   Fach\-be\-reich Ma\-the\-ma\-tik und In\-for\-ma\-tik,
   Philipps-Uni\-ver\-si\-t\"at Mar\-burg,
   Hans-Meer\-wein-Stra{\ss}e,
   D-35032~Mar\-burg, Germany.
   
   \nopagebreak
\textit{E-mail address:} \texttt{schmitzd@mathematik.uni-marburg.de}
   
     \vskip .2cm
Henrik Sepp\"{a}nen,
Mathematisches Institut,
Georg-August-Universit\"at G\"ottingen,
Bunsenstra\ss e 3-5, 
D-37073 G\"ottingen,
Germany

 \nopagebreak
\textit{E-mail address:} \texttt{hseppaen@uni-math.gwdg.de}

%
%  \vskip .2cm
%   Stefano Urbinati,
%   Universit\`a degli Studi di Padova,
%   Dipartimento di Matematica,
%   ROOM 608, Via Trieste 63,
%   35121 Padova, Italy.
%
%   \nopagebreak
%\noindent   \textit{E-mail address:} \texttt{urbinati.st@gmail.com}

%*****************************************************************************


\begin{thebibliography}{99}\footnotesize\itemsep=0cm\parskip=0cm

\bibitem[AKL12]{akl} Anderson, D., K\"uronya, A., Lozovanu, V., \emph{Okounkov bodies of finitely generated
divisors}, IMRN 2013; doi: 10.1093/imrn/rns286

%\bibitem[B97]{B}
%Bauer, Th.,
%{\it Seshadri constants of quartic surfaces},
%{Math. Ann.} {\bf 309} (1997), 3, 475--481

\bibitem[B98]{b:abelian}
Bauer, Th.,
\emph{On the cone of curves of an abelian variety},
Am. J. Math. {\bf 120} (1998), 997--1006 

\bibitem[BKS04]{bks} Bauer, Th., K\"uronya, A., Szemberg, T., \emph{Zariski chambers, volumes, and stable base loci},	
J. reine angew. Math. {\bf 576} (2004), 209-233 

\bibitem[G11]{gon}
Gonz\'alez, J.,
{\it Okounkov bodies on projectivizations of rank two toric vector bundles},
Journ. Alg. {\bf 330} (2011), 1, 322--345 

%\bibitem[HK00]{hk}
%		Hu, Y., Keel, S.:
		%Mori Dream Spaces and GIT.
		%Michigan Math. J. 48 (2000).

\bibitem[KK12]{kk} Kaveh, K, Khovanskii, A. G, \emph{Newton-Okounkov bodies, semigroups of integral points, graded algebras and intersection theory},  Ann. of Math. (2) {\bf 176} (2012), no. 2, 925-978

\bibitem[KLM12]{KLM} K\"uronya, A., Lozovanu, V., Maclean, C., \emph{Convex bodies appearing as Okounkov
bodies of divisors}, Advances of Mathematics 229 (2012), no. 5, 2622-2639

%\bibitem[LT04]{LT04}
%Lauritzen, N., Thomsen, J.F., 
%{\it Line bundles on Bott-Samelson varieties}

\bibitem[LB92]{LB}
Lange, H., Birkenhake, Ch.: 
``Complex Abelian Varieties'', Springer, 1992

\bibitem[L04]{laz}
{Lazarsfeld, R.,} 
``Positivity in Algebraic Geometry I'', Springer, 2004

\bibitem[LM09]{lm} Lazarsfeld, R., Musta{\c{t}}{\u{a}}, M., \emph{Convex bodies associated to linear series}, Ann. Sci. \'Ec. Norm. Sup\'er. (4) {\bf 42} (2009), no. 5, 783-835

\bibitem[L-SS13]{dp} \L uszcz-\'Swidecka, P., Schmitz, D., \emph{Minkowski decomposition of Okounkov bodies on
surfaces},  preprint, 2013

\bibitem[Oka10]{oka} Okawa, S., \emph{On global Okounkov bodies of Mori dream spaces}, preprint, 2010

\bibitem[Ok96]{ok} Okounkov, A., \emph{Brunn-Minkowski inequality for multiplicities}, Invent. Math. {\bf 125} (1996), 405-411

\bibitem[P11]{pet} Petersen, L., \emph{Okounkov bodies of complexity-one T-varieties}, preprint, 2011

\bibitem[PSU13]{psu} Pokora, P., Schmitz, D., Urbinati, S., \emph{Minkowski decomposition and generators of the moving cone for toric varieties}, preprint, 2013

\bibitem[S12]{sep1}
Sepp\"anen, H.,
{\it Okounkov bodies for ample line bundles},
preprint, 2012 (v3),  arxiv.org/abs/1007.1915

\bibitem[S14]{sep}
Sepp\"anen, H., 
{\it Okounkov bodies for ample line bundles and applications to multiplicities 
for group represenations}, preprint, 2014

\end{thebibliography}
\end{document}